\documentclass[review]{elsarticle}

\usepackage{lineno,hyperref}
\modulolinenumbers[5]

\usepackage{amsmath}
\usepackage{dsfont} 
\usepackage{amssymb} 
\usepackage{graphicx,color}
\usepackage{extarrows}
\usepackage{mathrsfs} 
\usepackage{booktabs}
\usepackage{multirow} 
\usepackage{threeparttable}
\usepackage[justification=raggedright]{caption}
\usepackage{threeparttable}
\usepackage[misc]{ifsym}
\usepackage{ntheorem}

\newtheorem{thm}{Theorem}[section]
\newtheorem{lem}[thm]{Lemma}
\newtheorem{prop}[thm]{Proposition}

\newtheorem*{proof}{proof}[section]

\newtheorem{exmp}{Example}[section]

\newtheorem{rem}{Remark}

\makeatletter
\@addtoreset{equation}{section}
\makeatother

\begin{document}
\begin{frontmatter}

\title{\bf\Large Maximum Likelihood Estimation of Stochastic Differential Equations with Random Effects Driven by Fractional Brownian Motion}


\author{Min Dai\fnref{addr1}}
\ead{mindai@hust.edu.cn}

\author{Jinqiao Duan\fnref{addr2}}
\ead{duan@iit.edu}

\author{Junjun Liao\fnref{addr1}}
\ead{liaojunjun@hust.edu.cn}

\author{Xiangjun Wang\fnref{addr1}\corref{mycorrespondingauthor}}
\cortext[mycorrespondingauthor]{Corresponding author}
\ead{xjwang@hust.edu.cn}

\address[addr1]{School of Mathematics and Statistics, \\
Huazhong University of Sciences and Technology,Wuhan 430074, China \\}
\address[addr2]{Department of Applied Mathematics, Illinois Institute of Technology, Chicago, IL 60616, USA}

\begin{abstract}
  Stochastic differential equations and stochastic dynamics are good models to describe stochastic phenomena in real world. In this paper, we study $N$ independent stochastic processes $X_i(t)$ with real entries and the processes are determined by the stochastic differential equations with drift term relying on some random effects. We obtain the Girsanov-type formula of the stochastic differential equation driven by Fractional Brownian Motion through kernel transformation. Under some assumptions of the random effect, we estimate the parameter estimators by the maximum likelihood estimation and give some numerical simulations for the discrete observations. Results show that for the different $H$, the parameter estimator is closer to the true value as the amount of data increases. \\
\end{abstract}


\end{frontmatter}

\linenumbers

\section{Introduction}

With the development of big data era, statistical analysis of data becomes more and more important. How to model these data effeciently becomes our new challenge. As is known, the generalized linear model \cite{McCullagh} is based on the likelihood method for regression analysis of various output results. However, the dependence between variables and the variation of variables over time are not taken into account. So it is more appropriate to consider the mixed effect model. The mixed effect model is very popular now in biomedical field \cite{Davidian,Pinheiro}, and its theory has a good development in deterministic model (no systematic error) \cite{Breslow,Diggle,Lindstrom,Vonesh}. As deterministic ordinary differential equation models do not account for the noise that often occurs in the system, to understand the dynamic behavior of these noise components, researchers have begun to study stochastic differential equation models with mixed effects. We know that stochastic differential equation models have a wide range of applications in the fields of physics, chemistry, biology, communication and so on. Stochastic differential equation models of the mixed effects have also been studied \cite{Overgaard}. For example, maximum likelihood estimation of the stochastic differential equations with random effects has been discussed \cite{Delattre}. Maximum likelihood estimation is a common parameter estimation method, but the maximum likelihood estimation of random effect parameters cannot be directly obtained because the likelihood function is seldom expressed. Ditlevsen and De Gaetano \cite{Ditlevsen} have shown that we can obtain an explicit expression of parameters from the likelihood function for some special situations. For general mixed stochastic differential equations, the approximation has been proved possible \cite{Picchini}. Delattre et al. \cite{Delattre} have studied mixed-effects stochastic differential equations with only the drift term having the random effects and then derived the exact expression of likelihood.

Stochastic differential equation models are always driven by Brownian motion, there is a kind of nearly completely random phenomenon in nature and society, however, has incremental stability, self-similarity and self-correlation. Especially, in the frequency domain, its power spectral density basically conforms to the polynomial decay law of $1/f$ within a certain frequency range, so it is called $1/f$ family random process. Fractional Brownian Motion (FBM) is a kind of the most widely used models, which is proposed by Benoit Mandelbrot and Van Ness. In addition, FBM possess some properties so that it can be widely used to explain many phenomena. The standard Fractional Brownian Motion $W_t^H$ with fractal parameter $H\in(0,1)$ is a Gaussian process with continuous paths, where mean zero and correlation function is the following form
$$E[W_s^HW_t^H]=\frac{1}{2}(|t|^{2H}+|s|^{2H}-|t-s|^{2H}).$$
Particularly, FBM is not semimartingale nor Markovian process, but it can be simplified to a standard Brownian motion when $H=\frac{1}{2}$. Since numerous results on FBM have appeared \cite{Biagini,Mishura}, it is very meaningful to study stochastic differential equations with mixed effects driven by Fractional Brownian Motion.

In this paper, we study a kind of stochastic differential equations with random effects driven by Fractional Brownian Motion, where the drift term has random effects. Specifically, we discuss $N$ independent stochastic process $(X_i(t),t\geq{0}),~i=1,\cdots,N$ with real entries, the dynamics of which are determined by
\begin{equation}\label{eq:1.1}
dX_i(t)={b(X_i(t),\phi_i)dt}+{\sigma(t)dW_i^H(t)},~~X_i(0)=x^i,i=1,\cdots,N,
\end{equation}
where $W_1^H,\cdots,W_N^H$ are $N$ mutually independent Fractional Brownian Motion with $H\in(\frac{1}{2},1)$ and the random variables $\phi_1,\cdots, \phi_N$ are independent identically distributed. The $(W_1^H,\cdots,W_N^H)$ and $(\phi_1,\cdots, \phi_N)$ are independent and $x^1,\cdots,x^N$ are initial values. The diffusion term $\sigma(t)$ is a positive deterministic function. We suppose that the distribution of $\phi_1,\cdots, \phi_N$ are the same and it can be expressed as $g(\psi,\theta)d\nu(\psi)$ on $\mathbb{R}^d$, where $g(\psi,\theta)$ is a density and $\theta$ is a parameter. Assume that $\theta_0$ is the true value. Then we want to estimate the parameters $\theta$ from the observations $\{X_i(t),i=1,\cdots,N\}$ which are observed on $[0,T_i]$. To ensure the models \eqref{eq:1.1} are defined well, some hypotheses are introduced. We then study one-dimensional linear case, which is $b(x,\phi_i)=\phi_ib(x)$, where $\phi_i$ satisfies Gaussian distribution. In this paper, we have obtained an explicit expression for likelihood with parameter $\theta$ about this case and we will give a specific explanation concerning the sufficient statistics in section 3:
\begin{equation*}
\begin{split}
U_i=\int_0^T{\Bigg(\frac{d}{d\omega_t^H}\int_0^t{k_H(t,s)\frac{b(X_i(s))}{\sigma(s)}ds}\Bigg)}dZ_i(s), &\\
V_i=\int_0^T{\Bigg(\frac{d}{d\omega_t^H}\int_0^t{k_H(t,s)\frac{b(X_i(s))}{\sigma(s)}ds}\Bigg)^2}d\omega_s^H.&
\end{split}
\end{equation*}
All of the above are derived for continuous time, but in fact, we need to deal with discrete observations on $[0,T_i]$ because the data of continuous observations is hard to obtain. So, we discretize the random variables $U_i,V_i$ and in our simulations, we give a detailed explanation. Moreover, we show several specific examples with the stochastic processes of discrete observations.

The paper is organized as follows. In section $2$, we introduce some preliminary knowledge and the hypotheses of the model. In section $3$, we present the expression of the likelihood function of the stochastic differential equation model driven by Fractional Brownian Motion and derive the maximum likelihood estimation of Gaussian one-dimensional random effect. In section $4$, we present the impact of discretization on the estimators and a simulation study. In section $5$, we give some conclusions and discussions.

\section{Preliminary}

\subsection{Girsanov-type formula}
We know that Fractional Brownian Motion is not a semimartingale, and it can't be handled as a markov process, but fortunately it is Gaussian process with continuous paths. In particular, the standard FBM $W_t^H$ becomes a standard Brownian motion when $H=\frac{1}{2}$, so we wonder if we can somehow make it better basing on that Fractional Brownian Motion has zero quadratic variation at $H>\frac{1}{2}$. Norros et al \cite{Norros} demonstrated that we can use an integral transformation to convert Fractional Brownian Motion at this time into a martingale. Next, we will introduce the basic martingale and stochastic differential equation driven by Fractional Brownian Motion.

Let $(\Omega,\mathcal{F},P)$ be a measurable space, including the natural filtration $\{\mathcal{F}_t\}$ which meets the usual conditions and can be treated as the $P$-complete filtration. Next we think about under showed the stochastic integral equation
\begin{equation}\label{eq:2.2}
X_t=\int_0^t{b(X_s)ds}+\int_0^t{\sigma(s)dW_s^H},~~t\geq{0}.
\end{equation}
where $b(X_t)$ is $\mathcal{F}_t$-adapted and the diffusion term $\sigma(t)$ is a nonzero function.
Noting that the following integral
\begin{equation}\label{eq:2.3}
\int_0^t{\sigma(s)dW_s^H}
\end{equation}
is not a It$\rm \hat{o}$ stochastic integral, but only the natural integral which is a deterministic function about Fractional Brownian Motion \cite{Norros,Gripenberg}. Although $W_t^H$ is not a semimartingale, we can transformed it into martingale and then the natural filtration of FBM $W^H$ and the martingale $M^H$ are same. For $0<s<t$, we define a $\textit{fundamental Gaussian martingale}$ $M^H$ \cite{Norros}
\begin{equation}\label{eq:2.4}
M_t^H=\int_0^t{k_H(t,s)dW_s^H},~~t\geq{0},
\end{equation}
where
\begin{equation}\label{eq:2.5}
\begin{split}
k_H(t,s)=k_H^{-1}s^{\frac{1}{2}-H}(t-s)^{\frac{1}{2}-H},~~~
k_H=2H~ \Gamma(\frac{3}{2}-H) ~\Gamma(H+\frac{1}{2}).
\end{split}
\end{equation}
And the quadratic variation of $M^H$ is equal to
\begin{equation}\label{eq:2.6}
\begin{split}
\omega_t^H=\lambda_H^{-1}t^{2-2H},
\end{split}
\end{equation}
where $\lambda_H={2H~ \Gamma(3-2H) ~\Gamma(H+\frac{1}{2})}/{\Gamma(\frac{3}{2}-H)}.$

Then, we can use the stochastic integral concerning the process $M^H$ to express \eqref{eq:2.3} as below \cite{Kleptsyna}
$$\int_0^t{K_H^\sigma(t,s)dM_s^H},~~~ 0\leq{s}\leq{t},$$
where $K_H^\sigma(t,s)=-2H\frac{d}{ds}\int_s^t{\sigma(r)r^{H-\frac{1}{2}}(r-s)^{H-\frac{1}{2}}dr}$, while the derivative which is absolute continuity about the Lebesgue measure makes sense \cite{Samko}. Similarly, for the stochastic integral equation \eqref{eq:2.2}, we can define the following formula well when $b(X_t)/\sigma(t)$ is smooth \cite{Samko}
\begin{equation}\label{eq:2.9}
Q_H(t)=\frac{d}{d\omega_t^H}\int_0^t{k_H(t,s)\frac{b(X(s))}{\sigma(s)}ds},t\in[0,T],
\end{equation}
where the derivative of the formula is comprehended in terms of absolute continuity .

Then we build a $\textit{fundamental semimartingale}$ $Z$ concerning the process $X$ by the same method as $W_t^H$ and they have the same the natural filtration. We have \cite{Kle}
\begin{equation}\label{eq:2.10}
\begin{aligned}
Z_t&=\int_0^t{k_H(t,s)[\sigma(s)]^{-1}}dX_s\\
&=\int_0^t{Q_H(s)d\omega_s^H}+M_t^H, ~~~~~~t\in[0,T],
\end{aligned}
\end{equation}
Therefore, the following Girsanov-type formula about stochastic differential equation driven by Fractional Brownian Motion can be obtained \cite{Kleptsyna,Prakasa}
\begin{thm}\cite{Kleptsyna,Prakasa}
Assume the process $Q_H$ has sample paths belonging $P$-a.s. to $L^2([0,T],d\omega^H)$. We define
\begin{equation}\label{eq:2.13}
\Lambda_H(T)=\mathrm{exp}\Bigg(-\int_0^T{Q_H(t)dM_t^H}-\frac{1}{2}\int_0^t{Q_H^2(t)d\omega_t^H}\Bigg).
\end{equation}
Then if the $E(\Lambda_H(T))=1$ holds, we can show that $P^*=\Lambda_H(T)P$ is a probability measure.
\end{thm}

\subsection{Model and assumptions}
We discuss $N$ independent stochastic processes $(X_i(t), t\geq{0}), i=1,\cdots,N$ with real entries, the dynamics of which are determined by
\begin{equation} \label{eq:2.15}
\begin{aligned}
dX_i(t)=b(X_i(t),\phi_i)dt+\sigma(t)dW_i^H(t),~~~~~~X_i(0)=x^i, i=1,\cdots,N,
\end{aligned}
\end{equation}
where $W_1^H,\cdots,W_N^H$ are $N$ mutually independent Fractional Brownian Motion with $H\in(\frac{1}{2},1)$, and $\phi_1,\cdots,\phi_N$ are random variables on $(\Omega,\mathcal{F},\mathbb{P})$. Specially, the diffusion coefficient $\sigma(t)$ is a positive non-random function. The natural filtration $(\mathcal{F}_t,t\geq{0})$ admitting the usual conditions is expressed by $\mathcal{F}_t=\sigma(\phi_i,W_i^H(s),s\leq{t},i=1,\cdots,N)=\sigma(W_i^H(s),s\leq{t})\bigvee{\mathcal{F}_t^i}$,
with $W_i^H$ and $\mathcal{F}_t^i=\sigma(\phi_i,\phi_j, W_j^H(s), s\leq{t},j\neq{i})$ are independent. In addition, the $\phi_i$ are measurable by $\mathcal{F}_0$. We will then introduce assumption ensuring that the processes \eqref{eq:2.15} are well defined.

$(\mathrm{A1})$ Let $b(x,\psi):\mathbb{R}\times{\mathbb{R}^d}\rightarrow{\mathbb{R}}$ be $C^1$ satisfying
$$\exists{K}>0,\forall(x,\psi)\in{\mathbb{R}\times\mathbb{R}^d}, b^2(x,\psi)\leq{K(1+x^2+\psi^2)}.$$
For all $\psi$, then the following equation
\begin{equation} \label{eq:2.16}
dX_i^\psi(t)={b(X_i^\psi(s),\psi)}dt+\sigma(t)dW_i^H(t),~~~~~~X_i^\psi(0)=x^i,
\end{equation}
has a unique solution $(X_i^\psi(t),t\geq{0})$ and the solution is adapted to the filtration $(\mathcal{F}_t,t\geq{0})$ \cite{Mishura}. The distribution of $X_i^\psi(\cdot)$ is defined on $C(\mathbb{R}^+,\mathbb{R})$, which is the continuous function space on $\mathbb{R}^+$. Furthermore, for all $\psi$, $t\geq{0}$ and integer $k$,
\begin{equation} \label{eq:2.17}
\sup_{s\leq{t}}{\mathbb{E}[X_i^\psi(s)]^{2k}}<+\infty.
\end{equation}

\section{Parameter Estimation of Fractional Brownian Motion with Random Effects}

\subsection{Likelihood function with random effects}
In this section, the likelihood function with random effect based on the standard model are derived. Suppose that the function $(x(t),t\in[0,T_i])$ is defined on the space $C_{T_i}$ which has the $\sigma$-algebra $\mathcal{C}_{T_i}$. On the space $(C_{T_i},\mathcal{C}_{T_i})$, let $G_\psi^{x^i,T_i}$ denote the distribution of process $(X_i^\psi(t),t\in[0,T_i])$ which is given by \eqref{eq:2.16} under the condition (A1). Then, define the joint distribution of $(\phi_i,X_i(\cdot))$ as $P_\theta^i=g(\psi,\theta)d\nu(\psi)\otimes{G_\psi^{x^i,T_i}}$ on $\mathbb{R}^d\times{C_{T_i}}$ and on $(C_{T_i},\mathcal{C}_{T_i})$, we define the marginal distribution of $(X_i(t),t\in[0,T_i])$ as $G_\theta^i$. So we obtain a standard process $(\phi_i,X_i(\cdot))$ on $\mathbb{R}^d\times{C_{T_i}}$. If there is no special explanation, the following text is about this process. Finally, we will give the following hypotheses in order to construct the likelihood function better.

$(\mathrm{A2})$ For every $\psi$,$\psi^{'}$
$$G_\psi^{x^i,T_i}\Bigg(\int_0^{T_i}\Bigg({{\frac{d}{d\omega_t^H}{\int_0^t{k_H(t,s){\frac{b(X_i^\psi(s),\psi^{'})}{\sigma(s)}}}ds}}\Bigg)^2}d\omega_t^H<{+\infty}\Bigg)=1,~~~i=1,\cdots,N.$$

$(\mathrm{A3})$ For $f=\frac{\partial{b}}{\partial{\psi_j}}, j=1,\cdots,n$, there exist a positive constant $c$ and a nonzero constant $\gamma$ satisfying
$$\sup_{\psi\in{\mathbb{R}^d}}{|f(x,\psi)|}\leq{c(1+|x|^\gamma)}.$$

\begin{thm}\label{thm:3.1}
Suppose that the hypotheses $(A1)-(A3)$ hold and let $\psi_0\in{\mathbb{R}^d}$.
\begin{itemize}
\item Under the sense of absolutely continuous, the distributions $G_\psi^{x^i,T_i}$ and $G^i:=G_{\psi_0}^{x^i,T_i}$ have the following relationship
$$\frac{dG_\psi^{x^i,T_i}}{dG^i}(X_i)=L_{T_i}(X,\psi)=e^{\textit{l}_{T_i}(X_i,\psi)},$$
where $$\textit{l}_{T_i}(X_i,\psi)=\int_0^{T_i}({Q_{H,\psi}-{Q_{H,\psi_0}}})dZ_i(s)-\int_0^{T_i}\frac{1}{2}({Q_{H,\psi}^2-Q_{H,\psi_0}^2}){d\omega_s^H},$$
and $(X_i=X_i(s),s\leq{T_i})$ indicates the standard process which is given by $(X_i(s)(x)=x(s),s\leq{T_i})$ on $C_{T_i}$.
\item The function $\psi\rightarrow{L_{T_i}(X_i,\psi)}$ satisfies a continuous version $G^i$-a.s. and $(X_i,\psi)\rightarrow{L_{T_i}(X_i,\psi)}$ is measurable on $(C_{T_i}\times{\mathbb{R}^d},\mathcal{C}_{T_i}\otimes{\mathcal{B}(\mathbb{R}^d)})$.
\end{itemize}
\end{thm}
\begin{proof}
The first part is obviously true under the assumption of $(A1)-(A2)$ \cite{Lipster}.\\
The second part, we will prove that $L_{T_i}(X,\psi)$ is continuous about $\psi$. First of all, we have the following integral with given $X_i$,
\begin{equation}\label{eq:3.1}
\psi\rightarrow{\int_0^{T_i}{Q_{H,\psi}^2(s)d\omega_s^H}}.
\end{equation}
It is easy to obtain the continuity of \eqref{eq:3.1} by the continuity theorem of integrals and the assumptions.
Then, for the other part of $L_{T_i}(X,\psi)$
\begin{equation}\label{eq:3.2}
\psi\rightarrow{\int_0^{T_i}{Q_{H,\psi}(s)dZ_i(s)}},
\end{equation}
it can be divided into two parts $I_1(\psi)+I_2(\psi)$, where
$$I_1(\psi)={\int_0^{T_i}{Q_{H,\psi}(s)}{Q_{H,\psi_0}(s)}{d\omega_s^H}},$$ $$I_2(\psi)={\int_0^{T_i}{{Q_{H,\psi}(s)}\{dZ_i(s)-{Q_{H,\psi_0}(s)}{d\omega_s^H}}}\}.$$
There are some results as follows.

First, the $I_1(\psi)$ is similar to \eqref{eq:3.1}, so it's continuous. Second, using the Burkholder-Davis-Gundy inequality for $I_2(\psi)$, we have
\begin{align*}
\mathbb{E}_{G^i}(I_2(\psi)-I_2(\psi'))^{2k}&\leq{C_k\mathbb{E}_{G^i}}\Bigg[\int_0^{T_i}\Bigg(\frac{d}{d\omega_t^H}\int_0^t{k_H(t,s)\frac{(b(X_i(s),\psi)-b(X_i(s),\psi'))^2}{\sigma(s)}ds}\Bigg){d\omega_s^H}\Bigg]^k\\
&\leq{C_k}|\psi-\psi'|^{2k}{\mathbb{E}_{G^i}}\Bigg[\int_0^{T_i}\Bigg(\frac{d}{d\omega_t^H}\int_0^t{c^2K|(1+X_i(s)|^{2\gamma})}K_{\sigma_0^2}^2\Bigg){d\omega_s^H}\Bigg]^k\\
&\leq{C(k,\gamma)}|\psi-\psi'|^{2k}\Bigg[\int_0^{T_i}\Bigg(\frac{d}{d\omega_t^H}\int_0^t{(1+\mathbb{E}_{G^i}(|X_i(s)|^{2\gamma}))}\Bigg){d\omega_s^H}\Bigg]^k.
\end{align*}
By \eqref{eq:2.17} and choosing $2k\geq{d}$, the ${I_2(\psi)}$ is continuous $G^i$-a.s. under the Kolmogorov criterion.

Now, we have proved that ${L_{T_i}(X_i,\psi)}$ is continuous about $\psi$ with given $X_i$. Furthermore, we know that ${L_{T_i}(X_i,\psi)}$ is measurable about $X_i$ with every $\psi$. Obviously, we can proved that $(X_i,\psi)\rightarrow{L_{T_i}(X_i,\psi)}$ is measurable.     $\hfill\square$
\end{proof}

Without loss of generality, assuming that the $\psi_0$ in Theorem\eqref{thm:3.1} satisfies $b(x,\psi_0)\equiv{0}$. Then, the measure $G^i=G_{\psi_0}^{x^i,T_i}$ is the distribution of \eqref{eq:2.16} having no drift. So the Radon Nikodym (R-N) derivative ${L_{T_i}(X_i,\psi)}$ can be simplified as
\begin{equation}\label{eq:3.3}
{L_{T_i}(X_i,\psi)}=\mathrm{exp}\Bigg(\int_0^{T_i}{Q_H(s)dZ_i(s)}-\frac{1}{2}\int_0^{T_i}{Q_H^2(s)d\omega_s^H}\Bigg),
\end{equation}
where $Q_H(t)=\frac{d}{d\omega_t^H}\int_0^t{k_H(t,s)\frac{b(X_i(s),\psi)}{\sigma(s)}ds}$.

On the other hand, we define the distribution of $(\phi_i,X_i(\cdot)),i=1,\cdots,N$ as $P_\theta=\otimes_{i=1}^{N}{P_\theta^i}$ on the space $\prod_{i=1}^N{\mathbb{R}^d\times{C_{T_i}}}$ and on $C=\prod_{i=1}^N{C_{T_i}}$, we define the distribution of process $(X_i(t),t\in[0,T_i],i=1,\cdots,N)$ as $G_\theta=\otimes_{i=1}^{N}{G_\theta^i}$ through independence of the individuals. In the end, the exact likelihood function is obtained by calculating the density of $G_\theta$ with respect to $G=\otimes_{i=1}^{N}{G^i}$.

\begin{thm}\label{pr:3.2}
While the hypotheses $(A1)-(A3)$ hold.
\begin{itemize}
  \item The probability measure $G_\theta^i$ satisfies the following density with respect to $G^i$
  $$\frac{dG_\theta^i}{dG^i}(X_i)=\int_{\mathbb{R}^d}{L_{T_i}(X_i,\psi)g(\psi,\theta)d\nu(\psi)}:=\lambda_i(X_i,\theta).$$
  \item The distribution $G_\theta$ and $G$ on $C=\prod_{i=1}^N{C_{T_i}}$ satisfies
  $$\frac{dG_\theta}{dG}(X_1,\cdots,X_N)=\prod_{i=1}^{N}{\lambda_i(X_i,\theta)}.$$
  \item The process $(X_i(t),t\in[0,T_i],i=1,\cdots,N)$ has the exact likelihood below
  \begin{equation}\label{eq:3.4}
  \Lambda_N(\theta)=\prod_{i=1}^{N}{\lambda_i(X_i,\theta)},
  \end{equation}
\end{itemize}
\end{thm}
\begin{proof}
For a positive measurable function $F$ on the space $C_{T_i}$, we obtain
$$E_{G_\theta^i}(F(X_i))=E_{P_\theta^i}(F(X_i))=E_{P_\theta^i}[E_{P_\theta^i}(F(X_i)|\phi_i)].$$
Using theorem \eqref{thm:3.1} and $L_{T_i}(X_i,\psi)$, we then have
$$E_{P_\theta^i}[(F(X_i)|\phi_i=\psi)]=E_{G_\psi^{x_i,T_i}}(F(X_i))=E_{G^i}(F(X_i)L_{T_i}(X_i,\psi)).$$
Moreover, we can establish the following equality by the joint measurability of $L_{T_i}(X_i,\psi)$ with two variables and the Fubini theorem,
\begin{align*}
E_{G_\theta^i}(F(X_i))&=\int_{\mathbb{R}^d}{g(\psi,\theta)dv(\psi)\mathbb{E}_{G^i}(H(X_i)L_{T_i}(X_i,\psi))}\\
&=\mathbb{E}_{G^i}(H(X_i))\int_{\mathbb{R}^d}{g(\psi,\theta)dv(\psi)(L_{T_i}(X_i,\psi))}.
\end{align*}
Therefore, we can calculate the $\text{R-N}$ derivative
$$\frac{dG_\theta^i}{dG^i}(X_i)=\int_{\mathbb{R}^d}{L_{T_i}(X_i,\psi)g(\psi,\theta)d\nu(\psi)}:=\lambda_i(X_i,\theta).$$
This is the exact likelihood formula. $\hfill\square$
\end{proof}

\subsection{Parameter estimation}
For the sake of simplicity, we now think about the special form of \eqref{eq:2.15} with $b(x,\psi)=\psi{b(x)}$, where $\psi\in{\mathbb{R}}$ and  $b(\cdot)$ and $\sigma(\cdot)$ are known. At the same time, we reduce the hypothesis (A1)-(A2) and suppose that $b,\sigma$ admits linear growth. For every $\psi$, we suppose that $\int_0^{T_i}{Q_H^2(s)d\omega_s^H}<\infty$, $G_{\psi}^{x^i,T_i}$-a.s. In order to make the observed processes $(X_i(t),i=1,\cdots,N)$ on $[0,T_i]$ independently and identically distribution, we set $T_i=T$, $x^i=x$ while $i=1,\cdots,N$. Difine
\begin{equation}\label{eq:3.5}
\begin{split}
U_i=\int_0^T{\Bigg(\frac{d}{d\omega_t^H}\int_0^t{k_H(t,s)\frac{b(X_i(s))}{\sigma(s)}ds}\Bigg)}dZ_i(s),\\
V_i=\int_0^T{\Bigg(\frac{d}{d\omega_t^H}\int_0^t{k_H(t,s)\frac{b(X_i(s))}{\sigma(s)}ds}\Bigg)^2}d\omega_s^H.
\end{split}
\end{equation}
We obtain the following result
\begin{equation}\label{eq:3.6}
\lambda_i(X_i,\theta)=\int_{\mathbb{R}^d}{g(\psi,\theta)\mathrm{exp}\Bigg(\psi{U_i}-\frac{\psi^2}{2}V_i\Bigg)d\nu(\psi)}.
\end{equation}
In this work, we suppose that the $\phi_i$ in \eqref{eq:2.15} satisfies Gaussian distribution $\mathcal{N}(\mu,\sigma_0^2)$ and the exact likelihood is given as below.

\begin{thm}\label{thm:3.3}
Suppose that $g(\psi,\theta)d\nu(\psi)=\mathcal{N}(\mu,\sigma_0^2)$. We then have
$$\lambda_i(X_i,\theta)=\frac{1}{(1+\sigma_0^2V_i)^{\frac{1}{2}}}\mathrm{exp}\Bigg[{-\frac{V_i}{2(1+\sigma_0^2V_i)}}\Bigg(\mu-\frac{U_i}{V_i}\Bigg)^2\Bigg]\mathrm{exp}\Bigg(\frac{U_i^2}{2V_i}\Bigg).$$
Under $P_\theta^i$, while $X_i$ are given, then the conditional distribution of $\phi_i$  is $$\mathcal{N}\Bigg(\frac{\mu+\sigma_0^2U_i}{1+\sigma_0^2V_i},\frac{\sigma_0^2}{1+\sigma_0^2V_i}\Bigg).$$
Hence, taking the logarithm of \eqref{eq:3.4}, we obtain
\begin{equation}\label{eq:3.7}
\mathcal{L}_N(\theta)=-\frac{1}{2}{\sum_{i=1}^N}{\mathrm{log(1+\sigma_0^2V_i)}}-\frac{1}{2}{\sum_{i=1}^N}{\frac{V_i}{1+\sigma_0^2V_i}}\Bigg(\mu-\frac{U_i}{V_i}\Bigg)^2+{\sum_{i=1}^N}\frac{U_i^2}{2V_i}.
\end{equation}
We derive the derivatives of \eqref{eq:3.7}
$$\frac{\partial}{\partial{\mu}}\mathcal{L}_N(\theta)={\sum_{i=1}^N}\Bigg({\frac{U_i}{1+\sigma_0^2V_i}}-\mu{\frac{V_i}{1+\sigma_0^2V_i}}\Bigg),$$
$$\frac{\partial}{\partial{\sigma_0^2}}\mathcal{L}_N(\theta)=\frac{1}{2}{\sum_{i=1}^N}\Bigg[\Bigg({\frac{U_i}{1+\sigma_0^2V_i}}-\mu{\frac{V_i}{1+\sigma_0^2V_i}}\Bigg)^2-{\frac{V_i}{1+\sigma_0^2V_i}}\Bigg].$$
Then we know that while $\sigma_0^2$ is given, the estimated value of $\mu_0$ is
\begin{equation}\label{eq:3.8}
\hat{\mu}_N={\sum_{i=1}^N}{\frac{U_i}{1+\sigma_0^2V_i}}\Bigg/{\sum_{i=1}^N}{\frac{V_i}{1+\sigma_0^2V_i}}.
\end{equation}
And while $\theta_0=(\mu_0,\sigma_0^2)$ are unknown, the estimated value is determined by the following system
$$\hat{\mu}_N=\Bigg({\sum_{i=1}^N}{\frac{V_i}{1+\hat{\sigma}_0^2V_i}}\Bigg)^{-1}\Bigg({\sum_{i=1}^N}{\frac{U_i}{1+\hat{\sigma}_0^2V_i}}\Bigg),$$
$${\sum_{i=1}^N}\Bigg(\hat{\mu}_N-\frac{U_i}{V_i}\Bigg)^2{\frac{V_i^2}{(1+\hat{\sigma}_0^2V_i)^2}}={\sum_{i=1}^N}{\frac{V_i}{1+\hat{\sigma}_0^2V_i}}.$$
\end{thm}
\begin{proof}
We first calculate the joint density of $(\phi_i,X_i)$
$$exp\Bigg(\psi{U_i}-\frac{\psi^2}{2}V_i\Bigg)\times\frac{1}{\sigma\sqrt{2\pi}}exp\Bigg[-\frac{1}{2\sigma_0^2}(\psi-\mu)^2\Bigg].$$
Then we obtain $$E_i=-\frac{1}{2}\Bigg[\psi^2(V_i+\sigma_0^{-2})-2\psi(U_i+\sigma_0^{-2}\mu)\Bigg]-\frac{1}{2}\sigma_0^{-2}\mu^2.$$
Set $$m_i=\frac{U_i+\sigma_0^{-2}\mu}{V_i+\sigma_0^{-2}}=\frac{\mu+\sigma_0^2U_i}{1+\sigma_0^2V_i},~~~\omega_i^2=(V_i+\sigma_0^{2})^{-1}=\frac{\sigma_0^{2}}{1+\sigma_0^{-2}V_i}.$$
Given $X_i$, the random variable $\phi_i$ satisfies the Gaussian law $\mathcal{N}(m_i,\omega_i^2)$.
And through some calculation, we have
$$E_i=-\frac{1}{2\omega_i^2}(\psi-m_i)^2-\frac{1}{2}V_i(1+\sigma_0^2V_i)^{-1}(\mu-V_i^{-1}U_i)^2+\frac{1}{2}V_i^{-1}U_i^2.$$
$\hfill\square$
\end{proof}

\begin{rem}\label{re:3.1}
In particular case, when the effect $\phi_i\equiv{\mu_0}$ corresponding to $\sigma_0^2=0$, the expression of the estimated value $\mu_0$ is
\begin{equation}\label{eq:3.9}
\tilde{\mu}_N=\sum_{i=1}^N{U_i}\Bigg/\sum_{i=1}^N{V_i}.
\end{equation}
\end{rem}

\subsection{Properties}
Now, we set up the random variable as follows
\begin{equation}\label{eq:3.10}
\gamma_i(\theta)=\frac{U_i-\mu{V_i}}{1+\sigma_0^2{V_i}},~~~I_i(\sigma_0^2)=\frac{V_i}{1+\sigma_0^2{V_i}}.
\end{equation}
In fact, the $\gamma_i(\theta)$ and $I_i(\sigma_0^2)$ are independent and identically distributed under $G_\theta$, we then get
\begin{equation}\label{eq:3.11}
\frac{\partial}{\partial{\mu}}\mathcal{L}_N(\theta)=\sum_{i=1}^N{\gamma_i(\theta)},~~~
\frac{\partial}{\partial{\sigma_0^2}}\mathcal{L}_N(\theta)=\frac{1}{2}\sum_{i=1}^N{(\gamma_i^2(\theta)-I_i(\sigma_0^2))}.
\end{equation}
Obviously, $I_i(\sigma_0^2)$ is bounded because of $0<I_i(\sigma_0^2)\leq{{1}/{\sigma_0^2}}$. Further, we can obtain
\begin{prop}\label{pr:3.1}
For every $\mu\in{\mathbb{R}}$, and every $\theta=(\mu,\sigma_0^2)\in{\mathbb{R}\times{\mathbb{R}^+}}$,
$$E_\theta\Bigg(\exp\Bigg(u\frac{U_1}{1+\sigma_0^2{V_1}}\Bigg)\Bigg)<+\infty.$$
\end{prop}
\begin{proof}
Let $\gamma_1(\theta)=\gamma_1$, $I_1(\sigma_0^2)=I_1$ and $l(X_1,\theta)=\log{\lambda_1(X_1,\theta)}$. Set $\theta(u)=(\mu+u,\sigma_0^2)$. Using $(U_1-(\mu+u)V_1)^2$, we have
$$l(X_1,\theta(u))=l(X_1,\theta)+u\gamma_1-\frac{u^2}{2}I_1,$$
where $\frac{\partial}{\partial{{\mu}}}l(X_1,\theta)=\gamma_1$ and $\frac{\partial^2}{\partial{{\mu}^2}}l(X_1,\theta)=-I_1$. Then taking exponential simultaneously, we can obtain
\begin{equation}\label{eq:3.12}
\lambda_1(X_1,\theta)\exp(u\gamma_1)=\lambda_1(X_1,\theta(u))\exp\Bigg(\frac{u^2}{2}I_1\Bigg).
\end{equation}
We can integrate both sides of \eqref{eq:3.12} with respect to the measure $Q^1$ on account of $I_1\leq{1/\sigma_0^2}$,
$$E_\theta{\exp(u\gamma_1)}=E_{\theta(u)}\exp\Bigg(\frac{u^2}{2}I_1\Bigg)\leq{\exp\Bigg(\frac{u^2}{2\sigma_0^2}\Bigg)}<{+\infty}.$$
Then, because of $u\mu\leq{(u+\mu)^2/4}$, we have
$$E_\theta\Bigg(\exp\Bigg(u\frac{U_1}{1+\sigma_0^2V_1}\Bigg)\Bigg)\leq{E_\theta{\exp(u\gamma_1)\exp\Bigg(\frac{(u+\mu)^2}{4\sigma_0^2}\Bigg)}}<{+\infty}.$$ $\hfill\square$
\end{proof}

\section{Discrete Data and Numerical Simulation}
In this section, we will do some numerical simulations to the model. Since we can only get discrete data from the simulation, let us start with a simply discussion of the discrete case. Suppose that we obtain the process $X_i(t)$ by observing at times $t_k^n=t_k=k\frac{T}{n}, k=1,2,\cdots,n+1$ simultaneously. We give the calculation of $Q_H(t)$ and $Z(t)$ to establish estimators $\hat{\theta}_N^{(n)}$
\begin{align*}
Q_H(t_k)=\frac{{\int_0^{t_k}{k_H(t_k,s)\frac{b(X(s))}{\sigma(s)}}ds}-{\int_0^{t_{k-1}}{k_H(t_{k-1},s)\frac{b(X(s))}{\sigma(s)}}ds}}{\omega^H(t_{k})-\omega^H(t_{k-1})},
\end{align*}
\begin{align*}
Z(t_k)\approx\sum_{i=2}^{k-1}{k_H^{-1}{t_i}^{\frac{1}{2}-H}(t_k-t_i)^{\frac{1}{2}-H}{\sigma(t_i)}^{-1}(X_{t_{i+1}}-X_{t_i})},
\end{align*}
where $${\int_0^{t_k}{k_H(t_k,s)\frac{b(X(s))}{\sigma(s)}}ds}\approx\sum_{i=2}^{k-1}{k_H^{-1}{t_i}^{\frac{1}{2}-H}(t_k-t_i)^{\frac{1}{2}-H}{b(X(t_i))}{\sigma(t_i)}^{-1}({t_{i+1}}-{t_i})}.$$

Then the random variables $U_i,V_i,i=1,\cdots,N$ are replaced by
\begin{equation}\label{eq:4.1}
U_i^n=\sum_{k=1}^{n}{Q_H(k)(Z_i(t_{k+1})-Z_i(t_k))},
\end{equation}
\begin{equation}\label{eq:4.2}
V_i^n=\sum_{k=1}^{n}{Q_H^2(k)(\omega^H(t_{k+1})-\omega^H(t_k))}.
\end{equation}
Combining with the log-likelihood \eqref{eq:3.7}, we can get the following expression of the difference $U_i-U_i^n$ and $V_i-V_i^n.$
\begin{lem}\label{lem:1}
Suppose that $b(X_t)$, $\sigma(t)$ and $b(X_t)/\sigma(t)$ are Lipschitz and $b(X_t)/\sigma(t)$ is bounded. For every $p\geq{1}$, there exists a constant $C$, then
$$\mathbb{E}_{\theta_0}(|V_i-V_i^n|^p+|U_i-U_i^n|^p)\leq{\frac{C}{n^{\frac{p}{2}}}}, ~~~i=1,\cdots,N.$$
\end{lem}
From Lemma \eqref{lem:1}, we derive
\begin{prop}
If $n\rightarrow{+\infty}$, then ${\hat{\theta}}_N-{\hat{\theta}}_N^{(n)}=o_{p_{\theta_0}}(1)$. If $\frac{n}{N}\rightarrow{+\infty}$ corresponding $n=n(N)\rightarrow{+\infty}$, then $\sqrt{N}{\hat{\theta}}_N-{\hat{\theta}}_N^(n)=o_{p_{\theta_0}}(1).$
\end{prop}

Next, some examples of simulations will be given. Assume that $b(x,\psi)=\sum_{j=1}^{d}{\psi^j}b^j(x)$ with $d=1$ or $d=2$.
\subsection{When $b(x),\sigma(t)$ are both constants}
Consider $b(x)$ and $\sigma(t)$ are both constants, and suppose that $b=c\sigma$ with $c\neq{0}$ known. We then have
\begin{equation}\label{eq:4.3}
U_i=\int_0^T{\Bigg(\frac{d}{d\omega_t^H}\int_0^t{ck_H(t,s)ds}\Bigg)}dZ_i(s),~~~~~~
V_i=\int_0^T{{\Bigg(\frac{d}{d\omega_t^H}\int_0^t{ck_H(t,s)}ds\Bigg)}^2}d\omega_s^H.
\end{equation}
The estimate of $(\mu_0,\omega_0^2)$ can be expressed explicitly as
$$\hat{\mu}_N=\frac{\sum_{i=1}^{N}{U_i}}{\sum_{i=1}^{N}{V_i}}=\frac{\bar{U}_N}{\bar{V}_N},~~~~~~~~\hat{\sigma}_N^2=\frac{1}{\bar{V}_N^2}\Bigg(\frac{1}{N}{\sum_{i=1}^{N}(U_i-\bar{U}_N)^2-\bar{V}_N}\Bigg).$$
And we can see from the above that no matter what the expression of the drift $b(\cdot)$ , the distribution of the estimated value is both the same when $b(x)$ is a constant times $\sigma(t)$.

\begin{exmp}\label{exmp:1}
Consider a random effect Fractional Brownian Motion
$$dX_i(t)=\phi_idt+\sigma{dW_i^H(t)},~~~~X_i(0)=x_0,$$
where $b(X_t)=1$ and $\phi_i\sim\mathcal{N}(\mu,\sigma_0^2).$ We think about the different Hurst index: $H=0.7$ and $H=0.9$, and the two sets: $(\mu=1,\sigma_0^2=1)$ and $(\mu=-1,\sigma_0^2=1)$. For all numerical situation of this example on $[0,T]$, we take $\sigma=1$ and the step-size $\Delta{t}=0.001$. Table 1 and table 2 show the results when $H=0.7$ and $H=0.9$, respectively. The results indicate the accuracy of the two parameter estimators $\mu$ and $\sigma$, we can see that as $N$ and $T$ increase, the estimated mean and variance become better.
\end{exmp}

\begin{table*}
\centering
\caption{$\textit{Example 4.1}$: The estimate value
 $\hat{\mu}_N$ and $\hat{\sigma}_0^2$ are calculated from multiple data sets when $H=0.7$.}
\renewcommand\tabcolsep{25.0pt} 
\begin{tabular}{cccc} \toprule
            \multirow{2}{*}{True value} &
            \multicolumn{1}{c}{N=30, T=5} & \multicolumn{1}{c}{N=50, T=5 } & \multicolumn{1}{c}{N=50, T=8} \\ \cline{2-4}
                            & Estimated value   & Estimated value  & Estimated value   \\ \hline
            $\mu=1$         & 1.0037  & 1.0386   & 1.0826    \\
            $\sigma_0^2=1$  & 0.6266  & 0.8379   & 0.9046    \\
            $\mu=-1$        & -0.8684  & -0.9856   & -0.8686     \\
            $\sigma_0^2=1$  & 0.7857  &  0.7523  &  0.9687  \\
            \bottomrule
  \end{tabular}
  \label{tbl:table-example}
\end{table*}

\begin{table*}
\centering
\caption{$\textit{Example 4.1}$: The estimate value
 $\hat{\mu}_N$ and $\hat{\sigma}_0^2$ are calculated from multiple data sets when $H=0.9$.}
\renewcommand\tabcolsep{25.0pt} 
\begin{tabular}{cccc} \toprule
            \multirow{2}{*}{True value} &
            \multicolumn{1}{c}{N=30, T=5} & \multicolumn{1}{c}{N=50, T=5 } & \multicolumn{1}{c}{N=50, T=8} \\ \cline{2-4}
                            & Estimated value   & Estimated value  & Estimated value   \\ \hline
            $\mu=1$         & 0.8284  & 0.8911  & 0.9495    \\
            $\sigma_0^2=1$  & 0.6569  & 0.7064  & 0.7404     \\
            $\mu=-1$        & -1.0965  & -0.9787  & -0.9812     \\
            $\sigma_0^2=1$  & 0.6368  &  0.6849  &  0.7382   \\
            \bottomrule
  \end{tabular}
  \label{tbl:table-example}
\end{table*}

\subsection{General case}
\begin{exmp}
Consider a Fractional Ornstein-Uhlenbeck-type process
$$dX_i(t)=(\phi_i+\phi_iX_i(t))dt+\sigma{dW_i^H(t)},~~~~X_i(0)=x_0,$$
where $b(X_t)=x+1$ and $\phi_i\sim\mathcal{N}(\mu,\sigma_0^2).$ Similar to example \eqref{exmp:1}, we give different $H$ and two different parameter sets: $(\mu=1,\sigma_0^2=1)$ and $(\mu=-5,\sigma_0^2=1)$. For all situation of this example on $[0,T]$, we take $\sigma=1$ and the step-size $\Delta{t}=0.001$. Specific consequences are shown in table 3 and table 4. We can see that as $N$ and $T$ increase, the estimated mean and variance are closer to the true value. In particular, we choose the initial value $x_0$ as far away from zero as possible, because when $x_0=0$, the estimation of the mean and variance will be greatly affected due to computational problems.
\end{exmp}

\begin{table*}
\centering
\caption{$\textit{Example 4.2}$: The estimate value $\hat{\mu}_N$ and $\hat{\sigma}_0^2$ are calculated from multiple data sets when $H=0.7$.}
\renewcommand\tabcolsep{25.0pt} 
\begin{tabular}{cccc} \toprule
            \multirow{2}{*}{True value} &
            \multicolumn{1}{c}{N=30, T=5} & \multicolumn{1}{c}{N=50, T=5 } & \multicolumn{1}{c}{N=50, T=8} \\ \cline{2-4}
                            & Estimated value   & Estimated value  & Estimated value   \\ \hline
            $\mu=5$         & 4.9794  & 5.1294  & 5.0827     \\
            $\sigma_0^2=1$  & 0.9654  & 0.9689  & 0.9754     \\
            $\mu=-1$        & -0.9845  & -1.0048  & -1.0435     \\
            $\sigma_0^2=1$  & 0.9455  & 1.0956  & 0.9661      \\
            \bottomrule
  \end{tabular}
  \label{tbl:table-example}
\end{table*}

\begin{table}
\centering
\caption{$\textit{Example 4.2}$: The estimate value $\hat{\mu}_N$ and $\hat{\sigma}_0^2$ are calculated from multiple data sets when $H=0.9$.}
\renewcommand\tabcolsep{25.0pt} 
\begin{tabular}{cccc} \toprule
            \multirow{2}{*}{True value} &
            \multicolumn{1}{c}{N=30, T=5} & \multicolumn{1}{c}{N=50, T=5 } & \multicolumn{1}{c}{N=50, T=8} \\ \cline{2-4}
                            & Estimated value   & Estimated value  & Estimated value   \\ \hline
            $\mu=5$         & 5.0580   & 5.0024  & 4.9759     \\
            $\sigma_0^2=1$  & 0.9535   & 0.9841  & 1.0584     \\
            $\mu=-1$        & -1.1984   & -1.0349  & -0.9398     \\
            $\sigma_0^2=1$  & 1.0625   & 1.0499  & 0.9959     \\
            \bottomrule
  \end{tabular}
  \label{tbl:table-example}
\end{table}

\section{Conclusions and Discussions}
In this paper, we have discussed the maximum likelihood estimation of stochastic differential equations with random effects driven by Fractional Brownian Motion. For the special case, we have derived an exact expression of the parameters from the likelihood function and the special case is that the drift term is correlated with the random effects linearly. In addition, some numerical simulations have been finished. Only the one-dimensional case is discussed in this work, however, we could theoretically derive for the multidimensional case, although it might be more complicated.

There are some interesting extensions for this study. For example, we can consider other process-driven stochastic differential equations with random effects existing in every aspect of life. On the other hand, it is of interest to consider the diffusion term with random effects. This project is in progress.

\section*{Acknowledgements}
We would like to thank  Yubin Lu, Xiaoli Chen and Xiujun Cheng for discussions about computation. This work was supported by the NSFC grants 11531006, 11801192, 11771449 and NSF grant 1620449.

\end{document}